\documentclass[11pt]{article}
\usepackage{hyperref}
\usepackage{amsfonts,mathrsfs,bbm,rawfonts,amsmath,amssymb,amsthm}
\usepackage{fullpage, setspace}
\usepackage{color}
\newtheorem{thm}{Theorem}
\newtheorem{lem}[thm]{Lemma}

\newtheorem{defn}[thm]{Definition}

\newtheorem{rem}[thm]{Remark}
\numberwithin{equation}{section}
\numberwithin{thm}{section}
\newcommand{\s}{\,\,\,\,\,}
\newcommand{\p}{\phi}
\newcommand{\n}{\nabla}
\newcommand{\R}{\mbox{Riem}}
\begin{document}

\title{Local Theory of Yang-Mills-Higgs-Schr\"{o}dinger Flow}

\author{Zonglin Jia\footnote{The corresponding author}}

\date{}

\maketitle

\begin{abstract}
In this article, we study two Hamiltonian type flows: Yang-Mills-Higgs-Schr\"odinger flow and $A$-Schr\"odinger flow. For the first one, we only obtain local existence. However, the uniqueness follows from classical tricks for the second one.
\\


\textbf{Subclass:}\quad 58E15, 35J50, 53C80.
\end{abstract}

\tableofcontents

\section{Introduction}
Throughout this present article, sometimes we assume that $(M,g,j)$ is a $2n$-dimensional closed K\"ahler manifold with a Riemannian metric $g$ and a complex structure $j$. Sometimes we assume that $(M,g)$ is an $m$-dimensional closed Riemannian manifold. 

$G$ is a compact Lie group with its Lie algebra $\mathfrak{g}$. $\pi: P\to M$ is a principal bundle with the structure group $G$. $(K,h,J)$ is a completed K\"ahler manifold with the metric $h$ and complex structure $J$ such that
\begin{eqnarray}\label{0}
J\xi W=\xi JW\s\mbox{for $W\in TK$ and $\xi\in\mathfrak{g}$,}
\end{eqnarray}
where
$$
\xi W:=\frac{\partial^2\left\{\exp(t\xi)\gamma(s)\right\}}{\partial t\partial s}\bigg|_{t=0,s=0}
$$
and $\gamma:(-\epsilon,\epsilon)\to K$ is any curve such that $(d\gamma/ds)(0):=W$. Moreover, the $G$-action preserves the metric $h$ on $K$. $P\times_GK$ and $P\times_{ad}\mathfrak{g}$ are the associated bundles corresponding to $P$ with $K$ and $\mathfrak{g}$ being the fibers respectively. 

\subsection{Two flows}
The Yang-Mills-Higgs-Schr\"odinger(YMHS) flow is defined as
\begin{eqnarray}\label{YMHS}
\left\{
\begin{aligned}
&J(\p)\partial_t\p=\n_A^*\n_A\p+\mu(\p)\n\mu(\p)\\
&j\partial_tA=D_A^*F_A+\p^*\n_A\p\\
&\p(0,\cdot):=\p_0,\s\s A(0,\cdot):=A_0
\end{aligned}
\right.
\end{eqnarray}
where $\p$ is a section on $P\times_GK$ and $A$ is a connection on $P$. The curvature $F_A$ is a $2$-form on $M$ taking values in $P\times_{ad}\mathfrak{g}$ which is defined as
$$
F_A:=dA+\frac{1}{2}[A,A]
$$
and $\n_A$ is the covariant derivative corresponding to $A$. The known map $\mu: P\times_GK\to P\times_{ad}\mathfrak{g}$ is called the moment map. $D_A$ is an exterior derivative on $P\times_{ad}\mathfrak{g}$ induced by $A$ and $D_A^*$ means its dual operator. Here $\p^*\n_A\p$ acts on any $B\in\Omega^1\left(P\times_{ad}\mathfrak{g}\right)$ by
$$
\left(\p^*\n_A\p,B\right):=\left(\n_A\p,B\p\right).
$$
The above geometric concepts can be referred to \cite{T,W}.
\\

The $A$-Schr\"odinger flow (ASF) is simpler than the previous one and given by
\begin{eqnarray}\label{ASF}
\left\{
\begin{aligned}
&J(\p)\partial_t\p=\n_A^*\n_A\p+\mu(\p)\n\mu(\p)\\
&\p(0,\cdot):=\p_0,
\end{aligned}
\right.
\end{eqnarray}
where $A$ is a known connection independent of time. Meanwhile, $\mu$ is the same as above. However, at this case the base manifold has only to be an $m$-dimensional Riemannian manifold.

\subsection{Background}
\s\s\textbf{Schr\"odinger flow}

Schr\"odinger flow origins from two intersting physical models: Landau-Lifshitz system and homogeneous Schr\"odinger equation. The former describes the evolution of ferromagnets(\cite{LL}); The latter is the fundamental equation of quantum mechanics telling us the states of particles without potentials.

\textbf{Yang-Mills-Higgs flow}

The Yang-Mills-Higgs (YMH) functional consists of the Yang-Mills functional, the kinetic energy functional and the Higgs potential energy functional. Since this functional appears naturally in the classical gauge theory, it has generated a lot of interests among both physicists and mathematicians during the past decades. For example, the Ginzburg-Landau equation in the superconductivity theory coincides with the variational equation of YMH functional. The critical points of the YMH functional are known as Yang-Mills-Higgs fields and the YMH functional is an appropriate Morse function to study the underlying spaces (\cite{AB, JT, P, S2, T2}). On the other hand, it is also well-known that the minimal YMH fields are the so-called symplectic vortices and their moduli space can be used to define invariants on symplectic manifolds with Hamiltonian actions (\cite{CGS, M}).

\textbf{Chern-Simons-Schr\"odinger equation}

The cousin of the Yang-Mills-Higgs-Schr\"odinger equation is the Chern-Simons-Schr\"odinger equation, which is often used to describe the behavior of particles moving in an electromagnetic field without energy dissipation. This equation is related to the quantum Hall theory, and details can be found in \cite{JP, JP2}, etc.

\subsection{History}

\s\,\,\,\,\textbf{Schr\"odinger flow}

Schr\"odinger flow was firstly suggested by Ding and Wang \cite{DW0}. In that paper, they get uniqueness and local smooth solutions defined on one dimensional circle. In 2001, They extended this result to the case of arbitrary dimensional base manifold \cite{DW}. In \cite{JW}, Jia and Wang studied the nonautonomous Schr\"odinger flows and got similar conclusions.

\textbf{Yang-Mills-Higgs flow}

If readers are interested in Yang-Mills-Higgs fields (namely the critical point of Yang-Mills-Higgs functional), they can refer to \cite{ASZ, CS, S2, S}. These articles are mostly about the existence, convergence, blow-up behaviors, isolated singularities and so on.

These literatures \cite{FH, SW} illustrate Yang-Mills-Higgs heat flow. Their common point is that the two references focus on weak solutions. While their difference is that the former considered flows on associated vector bundle and the latter studied ones on general associated bundle.

\textbf{Chern-Simons-Schr\"odinger equation}

Local or global well-posedness, global existence, scattering behaviors, blow-up properties of Chern-Simons-Schr\"odinger equation can be referred to \cite{BDS, D, H, LS, LST, OP, S3}.

\subsection{Main results of YMHS}
Our main results consist of two aspects: the base manifold $M$ is compact or non-compact. However, our strong solutions admit the regularities lower than the initial datum.

\begin{thm}\label{thm1.1}
Suppose $(M,g,j)$ is a $2n$-dimensional closed K\"ahler manifold. $G$ is a compact Lie group. $(K,h,J)$ is a completed K\"ahler manifold satisfing $(\ref{0})$. If $K$ is compact (or non-compact) and $\mu$ is smooth, the Cauchy problem $(\ref{YMHS})$ with initial data $\p_0\in W^{k+2,2}(M,K;D_C)$ and $A_0\in \mathscr{A}^{k+2,2}(M,P\times_{ad}\mathfrak{g})$, for $k\geq n+1$ (or $k\geq n+2$), admits a local strong solution $\p\in L^{\infty}([0,T],W^{k+2,2}(M,K;D_C))$, $A\in L^{\infty}([0,T],\mathscr{A}^{k,2}(M,P\times_{ad}\mathfrak{g}))$ such that 
$$
F_A\in L^{\infty}([0,T],W^{k+1,2}(\Omega^2(M,P\times_{ad}\mathfrak{g}))),
$$
where $T$ depends on 
$$
||F_{A_0}||_{W^{n+1,2}(\n_{A_0})}\left(\mbox{or}\,\,||F_{A_0}||_{W^{n+2,2}(\n_{A_0})}\right)\s\mbox{and}\s||\n_{A_0}\p_0||_{W^{n+1,2}(\n_{A_0})}\left(\mbox{or}\,\,||\n_{A_0}\p_0||_{W^{n+2,2}(\n_{A_0})}\right). 
$$
Moreover, if the initial data is smooth, the system admits a local smooth solution $(\p,A)$, where the existing time $T$ is as above. Under the assumption that $K$ is non-compact, we are able to find a compact set $K_c\subset K$ such that $\p([0,T]\times M)\subset P\times_GK_c$.
\end{thm}

\begin{thm}\label{thm1.2}
Suppose $M=\mathbb{R}^{2n}$ is the Euclidean space. $G$ is a compact Lie group. $(K,h,J)$ is a compact K\"ahler manifold satisfing $(\ref{0})$. If $\mu$ is smooth, the Cauchy problem $(\ref{YMHS})$ with initial datum $\p_0\in W^{k+2,2}(\mathbb{R}^{2n},K;D_C)$ and $A_0\in\mathscr{A}^{k+2,2}(\mathbb{R}^{2n},P\times_{ad}\mathfrak{g})$, for $k\geq n+1$, admits a local strong solution $\p\in L^{\infty}([0,T],W^{k+2,2}(\mathbb{R}^{2n},K;D_C))$, $A\in L^{\infty}([0,T],\mathscr{A}^{k,2}(\mathbb{R}^{2n},P\times_{ad}\mathfrak{g}))$ such that $F_A\in L^{\infty}([0,T],W^{k+1,2}(\Omega^2(\mathbb{R}^{2n},P\times_{ad}\mathfrak{g})))$, where $T$ depends on 
$$
||F_{A_0}||_{W^{n+1,2}(\n_{A_0})}\s\mbox{and}\s||\n_{A_0}\p_0||_{W^{n+1,2}(\n_{A_0})}. 
$$
Moreover, if
$$
\p_0\in\mathscr{W}^{\infty}:=\bigcap\limits_{k=2}^{\infty}W^{k,2}(\mathbb{R}^{2n},K;D_C)\s\mbox{and}\s A_0\in\mathscr{A}^{\infty}:=\bigcap\limits_{k=2}^{\infty}\mathscr{A}^{k,2}(\mathbb{R}^{2n},P\times_{ad}\mathfrak{g}),
$$
then the system admits a local solution
$$
\p\in C^{\infty}([0,T],\mathscr{W}^{\infty})\s\mbox{and}\s A\in C^{\infty}([0,T],\mathscr{A}^{\infty}),
$$
where the existing time $T$ is as above.
\end{thm}

\begin{rem}
The definitions of the above function spaces will be given in Subsection \ref{subsection2.6}.
\end{rem}

\subsection{Main results of ASF}
\begin{thm}
Suppose $(M,g)$ is an $m$-dimensional closed Riemannian manifold. $G$ is a compact Lie group. $(K,h,J)$ is a completed K\"ahler manifold satisfing $(\ref{0})$. If $K$ is compact (or non-compact) and $A$, $\mu$ are smooth, the Cauchy problem $(\ref{ASF})$ with initial data $\p_0\in W^{k,2}(M,K;D_A)$, for $k\geq\lfloor m/2\rfloor+2$ (or $k\geq\lfloor m/2\rfloor+3$) with $\lfloor q\rfloor$ denoting the largest integer not bigger than $q$, admits a local strong solution $\p\in L^{\infty}([0,T],W^{k,2}(M,K;D_A))$, where $T$ depends on 
$$
||\p_0||_{W^{\lfloor m/2\rfloor+2,2}(M,K;D_A)}\left(\mbox{or}\,\,||\p_0||_{W^{\lfloor m/2\rfloor+3,2}(M,K;D_A)}\right). 
$$
Moreover, if the initial data is smooth, the system admits a local smooth solution $\p$, where the existing time $T$ is as above. Under the assumption that $K$ is non-compact, we are able to find a compact set $K_c\subset K$ such that $\p([0,T]\times M)\subset P\times_GK_c$.
\end{thm}

\begin{thm}
Suppose $M=\mathbb{R}^m$ is the Euclidean space. $G$ is a compact Lie group. $(K,h,J)$ is a compact K\"ahler manifold satisfing $(\ref{0})$. If $\mu$ and $A$ are smooth, the Cauchy problem $(\ref{ASF})$ with initial datum $\p_0\in W^{k,2}(\mathbb{R}^m,K;D_A)$, for $k\geq\lfloor m/2\rfloor+2$, admits a local strong solution $\p\in L^{\infty}([0,T],W^{k,2}(\mathbb{R}^m,K;D_A))$, where $T$ depends on $||\p_0||_{W^{\lfloor m/2\rfloor+2,2}(\mathbb{R}^m,K;D_A)}$.
Moreover, if
$$
\p_0\in\mathscr{W}^{\infty}:=\bigcap\limits_{k=1}^{\infty}W^{k,2}(\mathbb{R}^m,K;D_A),
$$
then the system admits a local solution $\p\in C^{\infty}([0,T],\mathscr{W}^{\infty})$,
where the existing time $T$ is as above.
\end{thm}

Beforing stating the uniqueness, we define two function spaces:
$$
\mathscr{S}^{\infty}:=W^{2,2}(M,K;D_A)\bigcap\dot{W}^{1,\infty}(M,K;D_A)\bigcap\dot{W}^{2,\infty}(M,K;D_A)
$$
and
$$
\mathscr{S}^m:=W^{\lfloor m/2\rfloor+1,2}(M,K;D_A)\bigcap\dot{W}^{1,\infty}(M,K;D_A)\bigcap\dot{W}^{2,m}(M,K;D_A)
$$
\begin{thm}
Suppose $(M,g)$ is an $m$-dimensional completed manifold with bounded Ricci curvature, $K$ is a completed K\"ahler manifold with bounded geometry. If $\p_1, \p_2\in L^{\infty}([0, T], \mathscr{S}^{\infty})$ are two solutions to (\ref{ASF}) with the same initial map $\p_0\in\mathscr{S}^{\infty}$, then $\p_1=\p_2$ a.e. on $[0, T]\times M$.
\end{thm}

\begin{thm}
Suppose $m\geq3$ and $M$ is an $m$-dimensional completed manifold with bounded Riemannian curvature and positive injectivity radius, $K$ is a completed K\"ahler manifold with bounded geometry. If $\p_1, \p_2 \in L^{\infty}([0, T],\mathscr{S}^m)$ are two solutions to (\ref{ASF}) with the same initial map $\p_0\in\mathscr{S}^m$, then $\p_1 =\p_2$ a.e. on $[0, T]\times M$.
\end{thm}

\begin{rem}
A completed Riemannian manifold is said to have bounded geometry if it has positive injectivity radius and the Riemannian curvature tensor is bounded and has bounded derivatives.
\end{rem}

\subsection{The methods and tricks}

\s\,\,\textbf{Method of viscous disappearance}

The YMHS flows is a kind of Hamiltonian geometric flow, which is closely related to an energy functional $E(u)$ and its general definition is as following
$$
J(u)\partial_tu=\frac{\delta E}{\delta u}(u)
$$
which is equivalent to
$$
\partial_tu=-J(u)\frac{\delta E}{\delta u}(u).
$$
The conception of Hamiltonian geometric flow will be brought in later. Noting that the variation is generally a non-degenerative and positive operator, we approximate this flow by another system
\begin{eqnarray}\label{AS}
\partial_tu=-(\epsilon\mbox{Id}+J(u))\frac{\delta E}{\delta u}(u).
\end{eqnarray}
The advantage is that we transform a degenerative gradient flow into a negative one in order to use the classical theory of second order parabolic equations.

\textbf{DeTurck's trick}

In our present article, The Yang-Mills-Higgs functional is gauge invariant, which means $(\ref{AS})$ is also degenerative. To solve $(\ref{AS})$, we add a term to make the new system
\begin{eqnarray}\label{AS2}
\partial_tu=-(\epsilon\mbox{Id}+J(u))\frac{\delta E}{\delta u}(u)-\mathcal{E}u
\end{eqnarray}
be parabolic, where $\mathcal{E}$ is an operator of second order. As soon as we get a solution $v$ to $(\ref{AS2})$, it is natural to define a new map
$$
u(t,\cdot):=S(t)^*v(t,\cdot)
$$
for an operator $\mathcal{D}$ of first order and a gauge transformation $S$ dependent of time which satisfy
\begin{eqnarray*}
\left\{
\begin{array}{rl}
&\frac{dS}{dt}=S\circ\mathcal{D}v,\\
&S(0):=\mbox{Id}
\end{array}
\right.
\end{eqnarray*}
such that $u$ is a solution to $(\ref{AS})$.

\textbf{Geometric energy}

After obtaining the unique smooth solution $(\p_{\epsilon},A_{\epsilon})$ to the approximation system, We define an energy functional
$$
E_k:=\frac{1}{2}\sum\limits_{l=0}^k\left(||\n_{A_{\epsilon}}^lF_{A_{\epsilon}}||_{L^2}^2+||\n_{A_{\epsilon}}^{l+1}\p_{\epsilon}||_{L^2}^2\right)
$$
since of two reasons. The first one is that the evolutionary equation with respect to $\n_{A_{\epsilon}}^lF_{A_{\epsilon}}$ admits parabolic structure. Meanwhile the second one is that 
$$
\frac{d}{dt}||\n^k_{A_{\epsilon}}a_{\epsilon}||^2_{L^2}
$$ 
can be bounded by $E_{k+1}$ as soon as we get the uniform estimations with respect to $E_{k+1}$, where $a_{\epsilon}$ is the difference between $A_{\epsilon}$ and the reference connection $C$, i.e. $a_{\epsilon}:=A_{\epsilon}-C$.

\textbf{Uniform estimations}

The Gronwall's inequalities with respect to $E_k$ are divided into two cases: when $k\leq n+1$,
$$
dE_{n+1}/dt\leq C(n+1)(1+E_{n+1})^{\lambda(n)};
$$
and when $k\geq n+2$,
$$
dE_k/dt\leq C(k)(1+E_k)(1+E_{k-1})^{\lambda(k)},
$$
where $\lambda(k)$ are positive integers and $C(k)$ does not depend upon $\epsilon$. Such format ensures a uniform existing time $T$ for all the positive integers $k$.

\section{Preliminaries}
\subsection{Notations}
\s\s We adopt the Einstein summation convention.

Suppose $f:M\to N$ is a map. Given any subset $V\subset N$, $f^{-1}[V]$ denotes $\{x\in M: f(x)\in V\}$.

Let $U$ be a subset of a topological space $M$. $\overline{U}$ denotes the closure of $U$.

The local coordinates of the base manifold $M$ are denoted by $\{x^i\}$ and the counterpart of the fiber $K$ are denoted by $\{y^{\alpha}\}$. $\partial_i$ and $\partial_{\alpha}$ means $\partial/\partial x^i$ and $\partial/\partial y^{\alpha}$ respectively.

For a smooth manifold, $TM$ and $T^*M$ denote the tangent bundle and cotangent bundle respectively. $\Lambda^pM$ means the space of all the $p$-forms and $T^*M^{\otimes p}$ means
$$
\underbrace{T^*M\otimes T^*M\otimes\cdots\otimes T^*M}\limits_{\mbox{$p$-times}}.
$$

The point of $P\times_GK$ is denoted by $p\ast y$ with $p\in P$ and $y\in K$.

$P\times_{ad}\mathfrak{g}$ denotes the adjoint bundle of the principal $G$-bundle $P$ and $\Omega^p(P\times_{ad}\mathfrak{g})$ denotes the space of $P\times_{ad}\mathfrak{g}$-valued $p$-forms for $p\geq1$.

Denote the space of smooth connections on $P$ by $\mathscr{A}$, the set of smooth sections on $P\times_GK$ by $\mathscr{S}$.

$\R_M$ and $\R_K$ denote the Riemannian curvature tensors of the base manifold $M$ and the fiber $K$ respectively.

If $(K,h)$ is embedded isometrically into $\mathbb{R}^L$, then $\mathbb{A}$ denotes the second fundamental form of $K$.

$\n^{\mathfrak{g}}$, $\n^M$ and $\n^K$ denote the Levi-Civita connections on $\mathfrak{g}$, $M$ and $K$ respectively. $\n^{\mathfrak{g},l}$, $\n^{M,l}$ and $\n^{K,l}$ denote $(\n^{\mathfrak{g}})^l$, $(\n^M)^l$ and $(\n^K)^l$,  which mean the $l$-th covariant derivatives of $\n^{\mathfrak{g}}$, $\n^M$ and $\n^K$ respectively. $\n^{M*}$ denotes the dual of $\n^M$.

$V(B)$ denotes the vertical distribution for a fiber bundle $\pi: B\to M$ with the fiber $F$. If $F$ is endowed with a metric $h$, then $V(B)$ admits a canonical metric induced by $h$, which is also denoted by $h$.

For a vector bundle $\pi: E\to M$, $\Gamma(E)$ denotes the space of all the smooth sections and $\Omega^p(E)$ denotes the space of all the smooth $E$-valued $p$-form. In addition, if $E$ is endowed with a metric $\overline{h}$ and a covariant derivative $\overline{\n}$ which is compatible with $\overline{h}$, $W^{l,q}(E)$ and $W^{l,q}(\Omega^p(E))$ denote the $W^{l,q}$-Sobolev completions of $\Gamma(E)$ and $\Omega^p(E)$.

$\#$ denotes a multi-linear map with smooth coefficients.

Given a fiber bundle $B$ over $N$ and a smooth map $u: M\to N$ between two smooth manifolds $M$ and $N$, $u^*B$ denotes the pull-back bundle over $M$ and $du$ denotes the tangent map.

Given a connection $A$ on the principal bundle $P$, $\n_{A,Z}$ denotes $(\n_A)_Z$ and $\n_{A,i}$ denotes $\n_{A,\partial_i}$.

\subsection{Base manifold and principal bundle}
Given a tensor $\omega\in T^*M^{\otimes p}$ with $p$ being a positive integer, the $j$-action on $\omega$ is defined as
$$
(j\omega)(X_1,X_2\cdots,X_p):=\omega(jX_1,X_2\cdots,X_p).
$$
This means $j^2\omega=-\omega$, $g(j\omega,\omega)=0$ and $\n^M_Z(j\omega)=j\n^M_Z\omega$ for any $Z\in TM$.

Given an open neighborhood $U\subset M$, assume that $\Phi_U: \pi^{-1}[U]\to U\times G$ is a local trivialization. We define the so called unit local section $\sigma_U$ with respect to $\Phi_U$ as $\sigma_U(x):=\Phi_U^{-1}(x,e)$, where $e\in G$ is the unit.
\begin{lem}
Suppose $\pi: P\to M$ is a principal bundle with the structure group $G$. If $M$ and $G$ are compact, so is $P$.
\end{lem}
\begin{proof}
Suppose that on an open neighborhood $U\subset M$, $\pi^{-1}[U]$ is differentially homeomorphic to $U\times G$ via the local trivialization $\Phi_U$. It is easy to see that $\overline{V}\times G$ is compact for any open subset $V\subset U$. The compactness of $M$ implies that we can pick finite such neighborhoods $\{V_i: 1\leq i\leq m\}$ to cover $M$. The fact
$$
P=\bigcup\limits_{i=1}^m\Phi^{-1}_{U_i}\left(\overline{V}_i\times G\right)
$$
implies the wanted conclusion.
\end{proof}

\subsection{Global section on associated bundle}
It is well-known that a principal bundle may not admit any global section. It does if and only if the principal bundle is globally trivial. However, the coming theorem tells us that some special associated bundle admits at least one global section.
\begin{thm}\label{thm2}
If $K$ admits at least one point $o$ such that $go=o$ for any $g\in G$, then there is a global section $s: M\to P\times_GK$.
\end{thm}
\begin{proof}
The proof is easy and we leave it to the readers.
\end{proof}

The fiber $K$ who satisfies the condition of Theorem \ref{thm2} is called stable at the point $o$.

\begin{rem}
From now on, we always assume that the fiber $K$ is stable at some points.
\end{rem}

\subsection{Covariant derivative on associated bundle}
Assume that $U\subset M$ is any open neighborhood and $\p$ is a global section. Using a local unit section $\sigma$ we can write $\p|_U:=\sigma\ast u$ for a unique smooth map $u: U\to K$ and $A|_U:=A_idx^i$ with $A_i\in C^{\infty}(U,\mathfrak{g})$.

For any tangent vector $\xi\in\mathfrak{g}$, there is an infestimal action of $\xi$ on $K$, which generates a vector field $X_{\xi}$
\[ X_\xi(y) := \frac{d}{dt}\Big|_{t=0} exp(\xi t)y, ~\forall y \in K.\]

Firstly, employing this conception, we give the definition of covariant derivative acting on the section $\p\in\mathscr{S}$
$$
\left(\n_A\p\right)|_U:=du+X_{A_i}(u)\otimes dx^i.
$$

Secondly, we define the covariant derivative on the pull-back vertical distribution $\p^*V(P\times_GK)$ which is also denoted by $\n_A$
$$
\n_A:\Gamma(\p^*V(P\times_GK))\to\Gamma(\p^*V(P\times_GK)\otimes T^*M)
$$
$$
\n_AY:=\n^MY+\n^K_YX_{A_i}\otimes dx^i.
$$
Since of $(\ref{0})$, it is easy to verify
$$
\n^K_{JY}X_{\xi}=J\n^K_YX_{\xi}\s\mbox{for any $\xi\in\mathfrak{g}$ and}\s\n_AJ(\p)=0.
$$

Recall that $G$ respects the metric $h$. Hence $X_{\xi}$ is a Killing field and $\n^KX_{\xi}$ is anti-symmetric for any $\xi\in\mathfrak{g}$, i.e.
$$
h(Y,\n^K_{W}X_{\xi})+h(W,\n^K_{Y}X_{\xi})=0\s\mbox{for any $Y,W\in TK$}.
$$
The above conclusion tells us that $\n_A$ is compatible with $h$, i.e, $\n_Ah=0$. More details are referred to \cite{CS}. In the next, we shall give the specific expression of $\n_A^*\Phi$ for any $\Phi\in T^*M^{\otimes r}\otimes\p^*V(P\times_GK)$ with $r$ being a positive integer
$$
(\n_A^*\Phi)_{i_1\cdots i_{r-1}}=g^{kl}\n_{A,k}\Phi_{i_1\cdots i_{r-1}l}.
$$

Thirdly, we extend the above definition to $\p^*V(P\times_GK)$-valued $p$-forms. That is to say, we define
$$
\n_A: \Gamma\left(\p^*V(P\times_GK)\otimes\Lambda^pM\right)\to\Gamma\left(\p^*V(P\times_GK)\otimes\Lambda^pM\otimes T^*M\right)
$$
$$
\n_A(Y\otimes\omega):=\n_AY\otimes\omega+Y\otimes\n^M\omega\s\mbox{for $Y\in\p^*V(P\times_GK)$ and $\omega\in \Lambda^pM$}.
$$

Now there are two Laplace operators for the connection $A$. Namely, the Hodge Laplacian
$$\Delta_A = D_A^*D_A + D_AD_A^*$$
and the rough Laplacian $\nabla_A^*\nabla_A$. The well-known Weitzenb\"ock formula describes the difference
\begin{equation}\label{e:weitzenbock1}
  \nabla_A^*\nabla_A \Psi = \Delta_A \Psi+\R_M\#\Psi + F_A\#\Psi + \R_K\#d_A\phi\#d_A\phi\#\Psi
\end{equation}
for $\Psi\in\Gamma\left(\p^*V(P\times_GK)\otimes\Lambda^pM\right)$, where $d_A\phi:=\pi_A\circ d\phi$ and $\pi_A$ means the projection from $T(P\times_GK)$ to $V(P\times_GK)$. Locally $d_A\phi$ is exactly $du$.

\begin{rem}
If $K=\mathfrak{g}$ and the $G$-action on $\mathfrak{g}$ is the adjoint representation $ad$, then for a section $\psi$ on the associated vector bundle $P\times_{ad}\mathfrak{g}$, the covariant derivative still denoted by $\n_A$ is given by
$$
\n_A\psi:=dv+[a_i,v]\otimes dx^i,
$$
where on a neighborhood $V\subset M$, we have coincide  $\psi|_V$ with a smooth map $v: V\to\mathfrak{g}$.

The covariant derivative $\n_A$ on the pull-back vertical distribution $\psi^*V(P\times_{ad}\mathfrak{g})$ is then defined as
$$
\n_AY:=\n^MY+\n^{\mathfrak{g}}_Ya_i\otimes dx^i
$$
for any $Y\in\Gamma(\psi^*V(P\times_{ad}\mathfrak{g}))$.
\end{rem}

\subsection{Extrinsic format of associated bundle}
At first, we recall the following equivarant embedding theorem by Moore and Schlafly \cite{MS}.
\begin{thm}\label{et}
Suppose $(K,h)$ is a compact Riemannian manifold and $G$ is a compact Lie group which
acts on $K$ isometrically, then there exists an orthogonal representation $\rho: G\to O(L)$ and an isometric embedding $\iota: K\to\mathbb{R}^L$ such that $\iota(gy)=\rho(g)\iota(y)$ for any $g\in G$ and $y\in K$.
\end{thm}

Using the above representation, we define $\chi: \mathfrak{g}\to\mathfrak{o}(L)$ with $\chi(\xi):=d\rho(\xi)$ where $\mathfrak{o}(L)$ is the Lie algebra of $O(L)$, i.e. the set of all skew-symmetric $L\times L$ matrices. Thus we can write
$$
X_{\xi}(y)=\chi(\xi)y\s\mbox{for any $y\in K$ and $\xi\in\mathfrak{g}$}
$$
which implies
$$
\n^K_YX_{\xi}=\chi(\xi)Y-\mathbb{A}(y)(\chi(\xi)y,Y),
$$
where $\mathbb{A}$ is the second fundamental form of $K$ in $\mathbb{R}^L$.

\subsection{Sobolev space}\label{subsection2.6}
We need the vector bundle valued interpolation inequality which is dued to Ding and Wang.

\begin{thm}\label{DW}(\cite{DW})
Let $s\in\Gamma(E)$ for a smooth vector bundle $E$ endowed with a metric $\overline{h}$ and a compatible covariant derivative $\overline{\n}$ over a closed $m$-dimensional Riemannian manifold $(M,g)$. Then we have
$$
||\overline{\n}^js||_{L^p}\leq C||s||^a_{W^{l,q}(\overline{\n})}||s||^{1-a}_{L^r}
$$
with
$$
||s||^q_{W^{l,q}(\overline{\n})}:=\sum\limits_{i=0}^l||\overline{\n}^is||^q_{L^q}
$$
where $1\leq p,q,r\leq\infty$ and $j/l\leq a\leq 1$($j/l\leq a<1$ if $q=m/(l-j)\not=1$) are numbers such that
$$
\frac{1}{p}=\frac{j}{m}+\frac{1}{r}+a(\frac{1}{q}-\frac{1}{r}-\frac{l}{m})
$$
for the constant $C=C(M,j,l,q,r,a)$.
\end{thm}
\begin{rem}
The above constant $C$ is independent of the connection $\overline{\n}$.
\end{rem}

Fix a smooth connection $C\in\mathscr{A}$. Given any $A\in\mathscr{A}$, it is easy to check that $a:=A-C\in\Omega^1(M,P\times_{ad}\mathfrak{g})$. Define the affine sobolev space $$
\mathscr{A}^{l,q}(M,P\times_{ad}\mathfrak{g}):=\{C\}+W^{l,q}(\Omega^1(M,P\times_{ad}\mathfrak{g})),
$$
where $W^{l,q}(\Omega^1(M,P\times_{ad}\mathfrak{g}))$ is the completion of $C^{\infty}(\Omega^1(M,P\times_{ad}\mathfrak{g}))$ with respect to the norm $||a||_{W^{l,q}(\n_C)}$. It is easy to check that the above definition is independent of the choice of $C$.
\\

If $K$ is compact, applying the isometric embedding $\iota: K\to\mathbb{R}^L$ in Theorem \ref{et}, we embed $P\times_GK$ into the associated vector bundle $P\times_{\rho}\mathbb{R}^L$. Under the connection $C$, the covariant derivative of $P\times_{\rho}\mathbb{R}^L$ is denoted by $D_C$. The Sobolev space $W^{k,p}(M,K;D_C)$ is defined as
$$
W^{k,p}(M,K;D_C):=\left\{\mbox{$\p$ is a section on $P\times_GK$}: ||\p||^p_{W^{k,p}(M,K;D_C)}:=||\p||^p_{L^p}+||D_C\p||^p_{W^{k-1,p}(D_C)}<\infty\right\}.
$$
The homogeneous Sobolev space $\dot{W}^{k,p}(M,K;D_C)$ is defined as
$$
\dot{W}^{k,p}(M,K;D_C):=\left\{\mbox{$\p$ is a section on $P\times_GK$}: ||D_C^k\p||_{L^p}<\infty\right\}.
$$
\begin{defn}
Suppose that $Z$ is a topological space and $\{S_i\}$ is a family of compact subsets. Then $\cap S_i $ is called the minimal compact set.
\end{defn}

If $K$ is not compact, the Sobolev space $W^{k,p}(M,K;D_C)$ is defined as
\begin{eqnarray*}
W^{k,p}(M,K;D_C):&=&\big\{\mbox{$\phi$ is a section on $P\times_GK$}:\\
&&\mbox{$\exists$ minimal compact set $K_c\subset K$ such that $\phi(M)\subset P\times_G K_c$}, \\
&&||\phi||_{W^{k,p}(M,K;D_C)}:=||\phi||_{W^{k,p}(M,K_c;D_C)}<\infty\big\}
\end{eqnarray*}

Proposition 2.2 of \cite{DW} compares the difference of $||D_C\p||_{W^{k-1,2}(D_C)}$ and $||\n_C\p||_{W^{k-1,2}(\n_C)}$.
\begin{thm}\label{DW2}(\cite{DW})
Assume $k\geq n+1$ and that $K$ is compact. Then there exists a constant $C=C(K,k)$ such that for all smooth section $\p$ on $P\times_GK$,
\begin{eqnarray*}
||D_C\p||_{W^{k-1,2}(D_C)}\leq\sum\limits_{t=1}^k||\n_C\p||^t_{W^{k-1,2}(\n_C)}
\end{eqnarray*}
and
\begin{eqnarray*}
||\n_C\p||_{W^{k-1,2}(\n_C)}\leq\sum\limits_{t=1}^k||D_C\p||^t_{W^{k-1,2}(D_C)}.
\end{eqnarray*}
\end{thm}

\subsection{Yang-Mills-Higgs functional and the Hamiltonian geometric flow}
In order to illustrate the idea of YMHS flow, we firstly bring in the conception of Hamiltonian geometric flow. Suppose $u$ is a map taking values in a K\"ahler manifold with a complex structure $J$ and $E(u)$ is a kind of energy functional. Moreover, $\delta E/\delta u$ is the variation with respect to $E$. The Hamiltonian geometric flow is the following kind of partial differential equation
$$
J(u)\partial_tu=\frac{\delta E}{\delta u}(u).
$$

Secondly, we define the following Yang-Mills-Higgs energy functional
\begin{eqnarray*}
\mathcal{YMH}(\p,A):=\int_M|\n_A\p|^2+|F_A|^2+|\mu(\p)|^2.
\end{eqnarray*}
The critical point of $\mathcal{YMH}$ satisfies exactly the Yang-Mills-Higgs equations
\begin{eqnarray*}
\left\{
\begin{aligned}
&\n_A^*\n_A\p+\mu(\p)\n\mu(\p)=0\\
&D_A^*F_A+\p^*\n_A\p=0.
\end{aligned}
\right.
\end{eqnarray*}

\section{Local Existence to Approximation System}
We accept the method of viscous disappearance to approximate (\ref{YMHS}) by the following system
\begin{eqnarray}\label{3}
\left\{
\begin{aligned}
&\partial_t\p=-(\epsilon\mbox{Id}+J(\p))(\n_A^*\n_A\p+\mu(\p)\n\mu(\p))\\
&\partial_tA=-(\epsilon\mbox{Id}+j)(D_A^*F_A+\p^*\n_A\p)\\
&\p(0,\cdot):=\p_0,\s\s A(0,\cdot):=A_0,
\end{aligned}
\right.
\end{eqnarray}
where more details are referred to \cite{DW,JW}.  At beginning, we need the uniqueness and local existence of smooth solution to the approximation system.

It is not difficult to see that $(\ref{3})$ is also degenerate, since the Yang-Mills-Higgs functional is invariant under gange transformations. We are going to employ DeTurck's trick to transform the above equations into parabolic ones. Since our idea and process are the same as those of Theorem 3.1 in \cite{SW}, we only list the sketch and omit the details of computation. The readers can also refer to Section 4 of \cite{FH}.

\begin{thm}
For any $(\p_0,A_0)\in\mathscr{S}\times\mathscr{A}$, there is a positive time $T$ which depends upon $\p_0$ and $A_0$ and a unique smooth solution $(\p,A)$ to $(\ref{3})$ in $[0,T)\times M$ taking $(\p_0,A_0)$ as its initial data.
\end{thm}
\begin{proof}
To get existence, we consider two systems. The first one is as following
\begin{eqnarray}\label{DeTurck}
\left\{
\begin{array}{rl}
&\partial_t\psi+(\epsilon\mbox{Id}+J(\psi))\n_B^*\n_B\psi=-(\epsilon\mbox{Id}+J(\psi))\mu(\psi)\n\mu(\psi)+(\epsilon D_B^* b)\psi\\
&\partial_tb+(\epsilon\mbox{Id}+j)D_B^*F_B+\epsilon D_BD_B^*b=-(\epsilon\mbox{Id}+j)\psi^*\n_B\psi\\
&\psi(0,\cdot):=\phi_0,\s\s b(0,\cdot):=0
\end{array}
\right.
\end{eqnarray}
with $b:=B-A_0$. It is easy to verify that the above system is parabolic.

The second one is with respect to a family of gauge transformations $\{S(t)\}$
\begin{eqnarray*}
\left\{
\begin{array}{rl}
&\frac{dS}{dt}=-S\circ(\epsilon D_B^*b)\\
&S(0):=\mbox{Id}.
\end{array}
\right.
\end{eqnarray*}
At last, Let $(\phi(t),A(t)):=S(t)^*(\psi(t),B(t))$. It can be verified that this is the wanted solution.

For the uniqueness we reverse the above process. Suppose $(\p,A)$ is a solution to $(\ref{3})$, we solve the following parabolic equation to get a family of gauge transformations
\begin{eqnarray*}
\left\{
\begin{array}{rl}
&\partial_tS=-\epsilon D_A^*D_AS-D_A^*\circ S\{\epsilon(A-A_0)\}\\
&S(0):=\mbox{Id}.
\end{array}
\right.
\end{eqnarray*}
It is easy to check that if we define $(B(t),\psi(t)):=\{S(t)^{-1}\}^*(A(t),\p(t))$,  the next identity holds true
$$
S\circ\{\epsilon D_B^*(B-A_0)\}=\epsilon D_A^*D_AS+D_A^*\circ S\{\epsilon(A-A_0)\}
$$
and  $(B,\psi)$ solves $(\ref{DeTurck})$.
\end{proof}

\section{Uniform Estimations}\label{UE}
In order to omit some tedious computation, we assume that $M$ is a $2n$-dimensional flat torus with the usual metric from now on.

The smooth solution of $(\ref{3})$ is $(\p_{\epsilon},A_{\epsilon})$ with $a_{\epsilon}:=A_{\epsilon}-C$. For convenience we also denote them by $(\p,A,a)$ and $a$ throughout this section.

\subsection{Compact range}
Suppose $U_i\subset M$ ($1\leq i\leq m$) is an open neighborhood  and $\sigma_i: U_i\to P$ is a local unit section such that $\p_0|_{U_i}:=\sigma_i\ast u_i$ and
$$
M=\bigcup\limits_{i=1}^mU_i.
$$
Define
$$
\mathcal{R}_i:=\{y\in K: \mbox{dist}_K(y,u_i(M))<1\}
$$
and
$$
\mathcal{R}:=\bigcup\limits_{i=1}^m\mathcal{R}_i.
$$
From the compactness of $G\times\overline{\mathcal{R}}$ it follows that $G[\overline{\mathcal{R}}]$ and $P\times_GG[\overline{\mathcal{R}}]$ are compact where we have defined
$$
G[\mathcal{R}]:=\{gy: g\in G, y\in\mathcal{R}\}\s\mbox{and}\s G[\overline{\mathcal{R}}]:=\{gy: g\in G, y\in\overline{\mathcal{R}}\}.
$$
Furthermore, we define
$$
T_{\epsilon}:=\sup\{t>0: \p(t,M)\subset P\times_GG[\mathcal{R}]\}.
$$

\subsection{Geometric energy and evolutionary equation}
Define the energy functional
$$
E_k:=\frac{1}{2}\sum\limits_{l=0}^k\left(||\n_A^lF_A||_{L^2}^2+||\n_A^{l+1}\p||_{L^2}^2\right).
$$
Along $(\ref{3})$ $F_A$ satisfies
\begin{eqnarray*}
\partial_tF_A&=&-D_A\{(\epsilon\mbox{Id}+j)(D_A^*F_A+\p^*\n_A\p)\}=-(\epsilon\mbox{Id}+j)\{D_AD_A^*F_A+D_A(\p^*\n_A\p)\}\\
&=&-(\epsilon\mbox{Id}+j)\{\n^*_A\n_AF_A+g(F_A,\p)\}
\end{eqnarray*}
with
$$
g(F_A,\p):=F_A\#F_A+\n_A\p\#\n_A\p+\p^*F_A\p,
$$
where we have used the first Bianich identity and $(\ref{e:weitzenbock1})$.

In the next, we shall get the evolutionary equations with respect to $\n_A^lF_A$ and $\n_A^{l+1}\p$. Accepting the same process of equation $(4.30)$ in \cite{SW}, we obtain
\begin{eqnarray*}
\left\{
\begin{array}{rl}
&\partial_t\n_A^{l+1}\p=-(\epsilon\mbox{Id}+J(\p))\left\{\n_A^*\n_A^{l+2}\p+\mathcal{Q}_1^l(\n_A\p,F_A)+\frac{1}{2}\n_A^{l+1}\left\{\left(\n|\mu|^2\right)(\p)\right\}\right\}\\
&\partial_t\n_A^lF_A=-(\epsilon\mbox{Id}+j)\left\{\n_A^*\n_A^{l+1}F_A+\mathcal{Q}_2^l(\n_A\p,F_A)+\n_A^l\left(g(F_A,\p)\right)\right\},
\end{array}
\right.
\end{eqnarray*}
where $\mathcal{Q}^l_1, \mathcal{Q}^l_2$ are lower order terms depending on derivatives of $\n_A\p$ and $F_A$ up to order $l$. Moreover, we have the following estimations
$$
|\mathcal{Q}^l_1(\n_A\p,F_A)|\leq C_1(l)\sum|\n_A^{j_1}F_A|\cdots|\n_A^{j_r}F_A||\n_A^{1+j_{r+1}}\p|\cdots|\n_A^{1+j_{r+s}}\p|,
$$
where $C_1(l)$ is a positive constant depending on $\R_K$ and the indices satisfy
$$
0\leq j_1,\cdots,j_{r+s}\leq l;\s j_1 +\cdots+j_{r+s}=l;\s s\geq3.
$$
Furthermore, we can get
$$
|\n_A^l\left\{\left(\n|\mu|^2\right)(\p)\right\}|\leq C_2(l)\sum|\n_A^{i_1}\p|\cdots|\n_A^{i_p}\p|
$$
with
$$
i_1,\cdots,i_p\geq1; i_1+\cdots+i_p=l,
$$
where $C_2(l)$ is a positive constant depending on $\mu$. By similar idea, we also get
\begin{eqnarray*}
&&|\mathcal{Q}^l_2(\n_A\p,F_A)|+|\n_A^l(g(F_A,\p))|\\
&\leq&C_3(l)\sum\limits_{i=0}^l\left\{|\n_A^iF_A||\n_A^{l-i}F_A|+|\n_A^{i+1}\p||\n_A^{l-i+1}\p|\right\}\\
&&+C_3(l)\sum\limits_{0\leq j_1+j_2+j_3\leq l-2}|\n_A^{j_1}F_A||\n_A^{j_2+1}\p||\n_A^{j_3+1}\p|\\
&&+C_3(l)\sum\limits_{i=0}^{l-1}|\n^i_AF_A||\n_A^{l-i}\p|+C_3(l)|\n_A^lF_A|
\end{eqnarray*}
for a positive constant $C_3(l)$.

\begin{rem}
$C_1(l), C_2(l)$ and $C_3(l)$ depend on $P\times_GG[\mathcal{R}]$.
\end{rem}

\subsection{Controlling $||\n_A^lF_A||_{L^2}$}
Taking time derivative with respect to $||\n_A^lF_A||^2_{L^2}/2$ leads to
\begin{eqnarray*}
&&\frac{1}{2}\frac{d}{dt}||\n_A^lF_A||^2_{L^2}=\int\langle\n_A^lF_A,\partial_t\n_A^lF_A\rangle\\
&=&-\epsilon||\n_A^{l+1}F_A||_{L^2}^2-\int\langle\n_A^lF_A,(\epsilon\mbox{Id}+j)\mathcal{Q}_2^l(\n_A\p,F_A)\rangle-\int\langle\n_A^lF_A,(\epsilon\mbox{Id}+j)\n_A^l\left(g(F_A,\p)\right)\rangle\\
&\leq&C_4(l)\sum\limits_{i=0}^l\int|\n_A^lF_A||\n_A^iF_A||\n_A^{l-i}F_A|+C_4(l)\sum\limits_{i=0}^l\int|\n_A^lF_A||\n_A^{i+1}\p||\n_A^{l-i+1}\p|\\
&&+C_4(l)\sum\limits_{0\leq j_1+j_2+j_3\leq l-2}\int|\n_A^lF_A||\n_A^{j_1}F_A||\n_A^{j_2+1}\p||\n_A^{j_3+1}\p|\\
&&+C_4(l)\sum\limits_{i=0}^{l-1}\int|\n_A^lF_A||\n^i_AF_A||\n_A^{l-i}\p|+C_4(l)||\n_A^lF_A||^2_{L^2}
\end{eqnarray*}
Applying the same approaches of Lemma 3.2 and Lemma 3.3 in \cite{DW}, we can get the following estimations:

when $l\geq n+2$,
$$
\frac{1}{2}\frac{d}{dt}||\n_A^lF_A||^2_{L^2}\leq C_8(l)\left(1+||F_A||^2_{W^{l,2}(\n_A)}+||\n_A\p||^2_{W^{l,2}(\n_A)}\right)\left(1+||F_A||^2_{W^{l-1,2}(\n_A)}+||\n_A\p||^2_{W^{l-1,2}(\n_A)}\right)
$$
which is similar to $(3.24)$ in \cite{DW};

When $l\leq n+1$,
$$
\frac{1}{2}\frac{d}{dt}||\n_A^lF_A||^2_{L^2}\leq C_8(n)\left(1+||F_A||^4_{W^{n+1,2}(\n_A)}+||\n_A\p||^4_{W^{n+1,2}(\n_A)}\right).
$$
\subsection{Controlling $||\n_A^{l+1}\p||_{L^2}$}
Taking time derivative with respect to $||\n_A^{l+1}\p||^2_{L^2}/2$ leads to
\begin{eqnarray*}
&&\frac{1}{2}\frac{d}{dt}||\n_A^{l+1}\p||^2_{L^2}=\int\langle\n_A^{l+1}\p,\partial_t\n_A^{l+1}\p\rangle\\
&=&-\epsilon||\n_A^{l+2}\p||^2_{L^2}-\int\langle\n_A^{l+1}\p,(\epsilon\mbox{Id}+J(\p))\mathcal{Q}_1^l(\n_A\p,F_A)\rangle\\
&&-\frac{1}{2}\int\langle\n_A^{l+1}\p,(\epsilon\mbox{Id}+J(\p))\n_A^{l+1}\left\{\left(\n|\mu|^2\right)(\p)\right\}\rangle\\
&\leq&C_9(l)\sum\int|\n_A^{l+1}\p||\n_A^{j_1}F_A|\cdots|\n_A^{j_r}F_A||\n_A^{1+j_{r+1}}\p|\cdots|\n_A^{1+j_{r+s}}\p|\\
&&+C_9(l)\sum\int|\n_A^{l+1}\p||\n_A^{i_1}\p|\cdots|\n_A^{i_p}\p|
\end{eqnarray*}
Applying the same approaches of Lemma 3.2 and Lemma 3.3 in \cite{DW}, we can get the following estimations:

when $l\geq n+2$,
$$
\frac{1}{2}\frac{d}{dt}||\n_A^{l+1}\p||^2_{L^2}\leq C_{12}(l)\left(1+||F_A||^2_{W^{l,2}(\n_A)}+||\n_A\p||^2_{W^{l,2}(\n_A)}\right)\left(1+||F_A||^{\kappa(l)}_{W^{l-1,2}(\n_A)}+||\n_A\p||^{\kappa(l)}_{W^{l-1,2}(\n_A)}\right)
$$
which is similar to $(3.24)$ in \cite{DW};

When $l\leq n+1$, the same method yields
$$
\frac{1}{2}\frac{d}{dt}||\n_A^{l+1}\p||^2_{L^2}\leq C_{12}(n)\left(1+||F_A||^{\kappa(n)}_{W^{n+1,2}(\n_A)}+||\n_A\p||^{\kappa(n)}_{W^{n+1,2}(\n_A)}\right),
$$
where $\{\kappa(l)\}$ is a family of positive integers dependent of $l$.

\subsection{Controlling the energy}
Combining the above results, we obtain
\begin{eqnarray}\label{linear inequality}
dE_k/dt\leq C_{13}(k)(1+E_k)(1+E_{k-1})^{\lambda(k)}\s\mbox{for $k\geq n+2$}
\end{eqnarray}
and
$$
dE_{n+1}/dt\leq C_{13}(n)(1+E_{n+1})^{\lambda(n)}
$$
for some constants $C_{13}(k)$ and $\lambda(k)$. If let $f(t):=1+E_{n+1}(t)$, then we have
$$
df/dt\leq C_{13}(n)f^{\lambda(n)},\s f(0):=1+\frac{1}{2}\sum\limits_{l=0}^{n+1}\left(||\n_{A_0}^lF_{A_0}||_{L^2}^2+||\n_{A_0}^{l+1}\p_0||_{L^2}^2\right).
$$
It follows that there exists a positive time $T_{n+1}<T_{\epsilon}$ and 
$$
K_{n+1}:=K_{n+1}\left(||F_{A_0}||_{W^{n+1,2}(\n_{A_0})},||\n_{A_0}\p_0||_{W^{n+1,2}(\n_{A_0})}\right)
$$ 
satisfing
$$
E_{n+1}(t)\leq K_{n+1},\s\mbox{for all $t\in\left[0,T_{n+1}\right]$}.
$$
When $k\geq n+2$, The linear inequality $(\ref{linear inequality})$ with respect to $E_k$ implies inductively the bound 
$$
K_k:=K_k\left(||F_{A_0}||_{W^{k,2}(\n_{A_0})},||\n_{A_0}\p_0||_{W^{k,2}(\n_{A_0})}\right)
$$ 
for all $t\in\left[0,T_{n+1}\right]$, namely
$$
E_k(t)\leq K_k\s\mbox{for $t\in\left[0,T_{n+1}\right]$.}
$$

Using Theorem \ref{DW} and evolutionary equations yields
$$
||\partial_t\p(t)||_{L^{\infty}}\leq\bar{C}:=C_{14}(n)\max\{\sqrt{K_i}:0\leq i\leq n+2\}
$$
for $t\in\left[0,T_{n+1}\right]$ and some constant $C_{14}(n)$. Thus we have
$$
\mbox{dist}(\p(t,x),\p_0(x))\leq \bar{C}t.
$$
Let $T:=\min\left\{T_{n+1},(2\bar{C})^{-1}\right\}$. Along $[0,T]$ we get
$$
E_k(t)\leq K_k\s\mbox{for $k\geq n+2$,}
$$
where $K_k$ depends upon
$$
\sum\limits_{l=0}^k\left(||\n_{A_0}^lF_{A_0}||_{L^2}^2+||\n_{A_0}^{l+1}\p_0||_{L^2}^2\right).
$$
\begin{rem}
It is worthwhile pointing out that, if $K$ is compact, the existing time $T:=T_{n+1}$ depends on 
$$
||F_{A_0}||_{W^{n+1,2}(\n_{A_0})}\s\mbox{and}\s||\n_{A_0}\p_0||_{W^{n+1,2}(\n_{A_0})}.
$$
\end{rem}

\subsection{Controlling $||A-C||_{W^{k,2}(\n_C)}$ and $||\n_C\p||_{W^{k,2}(\n_C)}$}
Recall $a=A-C$. Given any $Y\in T^*M^{\otimes r}\otimes\p^*V(P\times_GK)$, it is not difficult to verify 
\begin{eqnarray}\label{formula}
\n_C^kY=\n_A^kY+\mathcal{Q}^k(Y,a)
\end{eqnarray}
where $\mathcal{Q}^k$ are lower order terms depending on derivatives of $Y$ and $a$ up to order $k-1$ and satisfy
$$
|\mathcal{Q}^k(Y,a)|\leq C_k\sum|\n_A^{i_1}Y|\cdots|\n_A^{i_r}Y||\n_A^{i_{r+1}}a|\cdots|\n_A^{i_{r+s}}a|
$$
with $0\leq i_1+\cdots+i_r+i_{r+1}+\cdots+i_{r+s}\leq k-1$, $r,s\geq1$ and $r+s\leq k+1$ for some constant $C_k$ depending upon $P\times_GK$.

Direct calculation gives
$$
\partial_t\n_A^ka=\n_A^k\partial_ta+\sum\limits_{i=1}^{k-1}\n_A^i\partial_ta\#\n_A^{k-i}a
$$
and
$$
\n_A^l\partial_ta=-(\epsilon\mbox{Id}+j)\left(\n_A^lD_A^*F_A+\p^*\n_A^l\p+\sum\limits_{p=1}^l\n_A^p\p\#\n_A^{l-p+1}\p\right),
$$
which means
\begin{eqnarray*}
||\n_A^l\partial_ta||_{L^r}\leq C_{15}\left(||\n_A^{l+1}F_A||_{L^r}+||\n_A^l\p||_{L^r}+\sum\limits_{p=1}^l||\n_A^p\p||_{L^s}||\n_A^{l-p+1}\p||_{L^t}\right)
\end{eqnarray*}
for all $r,s,t\in[1,\infty]$ with $1/r=1/s+1/t$.

In order to estimate $||\n_A^ka||_{L^2}$, we compute
\begin{eqnarray*}
&&\frac{1}{2}\frac{d}{dt}||\n_A^ka||^2_{L^2}
\leq C_{16}||\n^k_Aa||_{L^2}\left(||\n^k_A\partial_ta||_{L^2}+\sum\limits_{i=1}^{k-1}||\n^i_A\partial_ta||_{L^{p_i}}||\n_A^{k-i}a||_{L^{q_i}}\right)\\
&\leq&C_{17}||\n^k_Aa||_{L^2}\left(||\n_A^{k+1}F_A||_{L^2}+||\n_A^k\p||_{L^2}+\sum\limits_{p=1}^k||\n_A^p\p||_{L^s}||\n_A^{k-p+1}\p||_{L^t}\right)\\
&&+C_{17}\sum\limits_{i=1}^{k-1}||\n^k_Aa||_{L^2}||\n_A^{k-i}a||_{L^{q_i}}\left(||\n_A^{i+1}F_A||_{L^{p_i}}+||\n_A^i\p||_{L^{p_i}}+\sum\limits_{p=1}^i||\n_A^p\p||_{L^{s_i}}||\n_A^{i-p+1}\p||_{L^{t_i}}\right)
\end{eqnarray*}
with $s,t,p_i,q_i,s_i,t_i\in[2,\infty]$ to be determined such that $1/2=1/s+1/t$, $1/p_i+1/q_i=1/2$ and $1/p_i=1/s_i+1/t_i$. Theorem \ref{DW} tells us that we can find indices meeting the above conditions to ensure the following inequalities
$$
||\n_A^p\p||_{L^s}\leq C_{18}||\n_A\p||_{W^{k,2}(\n_A)}\s\mbox{and}\s||\n_A^{k-p+1}\p||_{L^t}\leq C_{18}||\n_A\p||_{W^{k,2}(\n_A)}
$$
$$
||\n_A^{k-i}a||_{L^{q_i}}\leq C_{18}||a||_{W^{k,2}(\n_A)}\s\mbox{and}\s||\n_A^{i+1}F_A||_{L^{p_i}}\leq C_{18}||F_A||_{W^{k,2}(\n_A)}
$$
$$
||\n_A^i\p||_{L^{p_i}}\leq C_{18}||\n_A\p||_{W^{k,2}(\n_A)}\s\mbox{and}\s||\n_A^p\p||_{L^{s_i}}\leq C_{18}||\n_A\p||_{W^{k,2}(\n_A)}
$$
and
$$
||\n_A^{i-p+1}\p||_{L^{t_i}}\leq C_{18}||\n_A\p||_{W^{k,2}(\n_A)}.
$$

To sum it up, it follows
\begin{eqnarray*}
\frac{1}{2}\frac{d}{dt}||a||^2_{W^{k,2}(\n_A)}\leq C_{19}(1+E_{k+1}^2)(1+||a||^2_{W^{k,2}(\n_A)}).
\end{eqnarray*}
Gronwall's Inequality tells us that there exists a smooth increasing function $f_k$ such that
$$
||a(t)||_{W^{k,2}(\n_A)}\leq f_k\left(t,E_{k+1},||A_0-C||_{W^{k,2}(\n_{A_0})}\right)
$$
for all $t\in[0,T]$.

At last of this subsection, we shall estimate $||a(t)||_{W^{k,2}(\n_C)}$ and $||\n_C\p||_{W^{k,2}(\n_C)}$. From $(\ref{formula})$ it follows
\begin{eqnarray*}
||\n_C^ka||_{L^2}&\leq&||\n_A^ka||_{L^2}+C_k\sum\Big|\Big||\n_A^{i_1}a|\cdots|\n_A^{i_t}a|\Big|\Big|_{L^2}\\
&\leq&||\n_A^ka||_{L^2}+C_k\sum||\n_A^{i_1}a||_{L^{p_{i_1}}}\cdots||\n_A^{i_t}a||_{L^{p_{i_t}}}
\end{eqnarray*}
with $0\leq i_1+\cdots+i_t\leq k-1$ and $1\leq t\leq k+1$, where $p_{i_1},\cdots,p_{i_t}$ are positive numbers to be determined such that
$$
\frac{1}{p_{i_1}}+\cdots+\frac{1}{p_{i_t}}=\frac{1}{2}.
$$
Theorem \ref{DW} tells us that we can pick indices $p_{i_1},\cdots,p_{i_t}$ such that
$$
||\n_A^{i_r}a||_{L^{p_{i_r}}}\leq C_k||a||_{W^{k-1,2}(\n_A)}\s\mbox{for $1\leq r\leq t$.}
$$
To sum it up, we are able to obtain an increasing function $\beta_k$ such that
$$
||a(t)||_{W^{k,2}(\n_C)}\leq\beta_k\left(t,E_{k+1},||A_0-C||_{W^{k,2}(\n_{A_0})}\right)
$$
for all $t\in[0,T]$.

Similar approaches are applied to 
$$
\n_C^{k+1}\p=\n_A^{k+1}\p+\mathcal{Q}^k(\n_A\p,a)\s\mbox{and}\s\n_C^kF_A=\n_A^kF_A+\mathcal{Q}^k(F_A,a).
$$
Hence we can get two increasing functions $\alpha_k$ and $\tau_k$ such that
$$
||\n_C\p(t)||_{W^{k,2}(\n_C)}\leq \alpha_k\left(t,E_k,||A_0-C||_{W^{k,2}(\n_{A_0})}\right)
$$
and
$$
||F_A(t)||_{W^{k,2}(\n_C)}\leq \tau_k\left(t,E_k,||A_0-C||_{W^{k,2}(\n_{A_0})}\right)
$$
for all $t\in[0,T]$.

\section{Proof of Theorem \ref{thm1.1}}
We only list the sketch since the trick to get local smooth or regular solution is the same as that of Theorem 1.1 in \cite{DW}.

\textbf{Step 1 of smoothness:}

Theorem \ref{DW2} tells us that we have already obtained
$$
||F_{A_{\epsilon}}(t)||_{W^{k+1,2}(\n_C)}+||D_C\p_{\epsilon}(t)||_{W^{k+1,2}(D_C)}+||a_{\epsilon}(t)||_{W^{k,2}(\n_C)}\leq\gamma_k
$$
for some constant 
$$
\gamma_k=\gamma_k\left(K_{k+1},||A_0-C||_{W^{k,2}(\n_{A_0})}\right)
$$ 
and all $t\in[0,T]$.

\textbf{Step 2 of smoothness:}

Fixing $k$,  we get $(\p_k,A_k)\in C^k$ by sending $\epsilon$ to $0$ and Sobolev Embedding Theorem.

\textbf{Step 3 of smoothness:}

Diagonal method implies that we are able to find $(\p,A)\in C^{\infty}$ such that for any $k$, $(\p_{\epsilon},A_{\epsilon})$ converges to $(\p,A)$ in $C^k$. Hence, $(\p,A)$ is a smooth solution to $(\ref{YMHS})$ with smooth initial data.
\\

\textbf{Step 1 of regularity:}

If the initial data $(\phi_0, A_0)$ lies in $W^{k+2,2}(M,K;D_C)\times\mathscr{A}^{k+2,2}(M,P\times_{ad}\mathfrak{g})$, we choose a family of smooth pair $(\phi^i_0,A^i_0)$ to approximate $(\phi_0, A_0)$ in $(W^{k+2,2}
(D_C), W^{k+2,2}(\n_C))$-norm. Let $(\phi_i,A_i)$ be the smooth solution to $(\ref{YMHS})$ taking $(\phi^i_0,A^i_0)$ as its initial data and existing for $[0, T_i]$.

\textbf{Step 2 of regularity:}

Repeating the process of Section \ref{UE}, we get a uniform lower bound $\hat{T}$ of $T_i$, i.e. $T_i\geq\hat{T}>0$ for large $i$ and uniform upper bounds of $(W^{k+2,2}(D_C), W^{k,2}(\n_C),W^{k+1,2}(\n_C))$-norms with respect to $(\phi_i,A_i,F_{A_i})$.

\textbf{Step 3 of regularity:}

Sending $i$ to $\infty$ gives a weak limit $(\phi, A,F_A)$. Classical trick implies that it is also a strong solution.

\section{Proof of Theorem \ref{thm1.2}}
We only list the sketch since the trick is the same as that of Theorem 1.2 in \cite{DW}.

\textbf{Step 1:}

Given any $(\p_0,A_0)$ lying in $W^{k+2,2}(\mathbb{R}^{2n},K;D_C)\times\mathscr{A}^{k+2,2}(\mathbb{R}^{2n},P\times_{ad}\mathfrak{g})$, we pick a family of smooth pair $(\p^i_0,A^i_0)$ supported by
$$
\underbrace{[-R_i,R_i]\times\cdots\times[-R_i,R_i]}\limits_{\mbox{$2n$-times}},
$$
which converges to $(\p_0,A_0)$ in $W^{k+2,2}(\mathbb{R}^{2n},K;D_C)\times\mathscr{A}^{k+2,2}(\mathbb{R}^{2n},P\times_{ad}\mathfrak{g})$-norm for any fixed $k$.

\textbf{Step 2:}

Let 
$$
\mathbb{T}_i:=\underbrace{(\mathbb{R}/[-R_i,R_i])\times\cdots\times(\mathbb{R}/[-R_i,R_i])}\limits_{\mbox{$2n$-times}}
$$
and regard $(\p^i_0,A^i_0)$ as a section defined on $\mathbb{T}_i$. Then we solve $(\ref{YMHS})$ taking $(\p^i_0,A^i_0)$ as its initial data on the base manifold $\mathbb{T}_i$ to get a family of smooth solutions $\{(\p^i,A^i)\}$.

\textbf{Step 3:}

The readers are able to obtain a local strong solution by repeating Step 2 and Step 3 of regularity.

\section{Local Well-Posedness to ASF}
Since ASF is similar to classical Schr\"odinger flow, the proofs of local existence and uniqueness are almost the same as those of \cite{DW} and \cite{SW2} respectively except for some terms with respect to the curvature $F_A$ which is now a known tensor. Hence, we omit details.

\section*{Acknowledgement}
The author is supported by National Natural Science Foundation of China (Foundation of Mathematical Tian Yuan) (Grant No. 12326345).

\vspace{1cm}
\noindent{Zonglin Jia}\\
Department of Mathematics and Physics, North China Electric Power University, Beijing,
China.\\
Email: 50902525@ncepu.edu.cn

\end{document}